\documentclass[12pt,reqno]{amsart}
\usepackage{amsmath,amsthm,amssymb,amsfonts,amscd}
\usepackage{mathrsfs}
\usepackage{bbm}
\usepackage{bbding}
\usepackage{graphicx,latexsym}
\usepackage{hyperref}
\usepackage{geometry}
\usepackage{color}
\usepackage{xcolor}

\numberwithin{equation}{section}

\theoremstyle{plain}
\newtheorem{theorem}{Theorem}

\newtheorem{corollary}[theorem]{Corollary}

\theoremstyle{definition}

\theoremstyle{remark}

\newcommand{\Z}{{\mathbb Z}} 
\newcommand{\Q}{{\mathbb Q}}
\newcommand{\R}{{\mathbb R}}

\renewcommand{\Re}{\operatorname{Re}}

\newcommand{\N}{\operatorname{N}}

\newcommand{\dd}{\mathrm{d}}

\newcommand{\Norm}{\operatorname{Norm}}

\newcommand{\E}{\mathbb E}

\newcommand{\ave}[1]{\left\langle#1\right\rangle} 

\newcommand{\Var}{\operatorname{Var}}

\makeatletter
\def\@tocline#1#2#3#4#5#6#7{\relax
  \ifnum #1>\c@tocdepth 
  \else
    \par \addpenalty\@secpenalty\addvspace{#2}%
    \begingroup \hyphenpenalty\@M
    \@ifempty{#4}{%
      \@tempdima\csname r@tocindent\number#1\endcsname\relax
    }{%
      \@tempdima#4\relax
    }%
    \parindent\z@ \leftskip#3\relax \advance\leftskip\@tempdima\relax
    \rightskip\@pnumwidth plus4em \parfillskip-\@pnumwidth
    #5\leavevmode\hskip-\@tempdima
      \ifcase #1
       \or\or \hskip 1em \or \hskip 2em \else \hskip 3em \fi%
      #6\nobreak\relax
    \hfill\hbox to\@pnumwidth{\@tocpagenum{#7}}\par
    \nobreak
    \endgroup
  \fi}
\makeatother

\begin{document}

\title[Gaussian primes in almost all narrow sectors]
{Gaussian primes in almost all narrow sectors}
\author{Bingrong Huang, Jianya Liu, and Ze\'ev Rudnick}
\address{School of Mathematical Sciences \\ Tel Aviv University \\ Tel Aviv \\ Israel}
\email{bingronghuangsdu@gmail.com}
\address{School of Mathematics \\ Shandong University \\ Jinan \\Shandong 250100 \\China}
\email{jyliu@sdu.edu.cn}
\address{School of Mathematical Sciences \\ Tel Aviv University \\ Tel Aviv \\ Israel}
\email{rudnick@post.tau.ac.il}
\thanks{The research of B.H. and Z.R. was supported by the European Research Council under the European Union's Seventh Framework Programme (FP7/2007-2013) / ERC grant agreement n$^{\text{o}}$ 786758;
the research of J.L. was supported by the National Science Foundation of China under Grant 11531008, the Ministry of Education of China under
Grant IRT16R43, and the Taishan Scholar Project of Shandong Province.}

\date{\today}

\begin{abstract}
 We study Gaussian primes lying in narrow sectors, and show that almost all such sectors contain the expected number of primes, if the sectors are not too narrow.
\end{abstract}

\maketitle

\section{Introduction} \label{sec:Intr}

Let $\mathfrak{p}$ be a prime ideal in the ring of Gaussian integers $\mathbb{Z}[i]$.
  Assuming the norm $\N(\mathfrak p)=p=1\bmod 4$ is a split prime, we can write $p=a^2+b^2$ with $\alpha=a+ib$ a generator of $\mathfrak p$, unique up to multiplication by one of the units $\pm1,\pm i $.  This gives us an angle $\theta_{\mathfrak p}$, so that $\alpha=a+ib=\sqrt{p}e^{i\theta_{\mathfrak p}}$, which is unique if we fix it to lie in $[0,\pi/2)$.

Hecke \cite{Hecke} proved that the angles $\theta_\mathfrak{p}$ are equidistributed in $[0,\pi/2)$ as $\mathfrak{p}$ varies over prime ideals of $\mathbb{Z}[i]$. This means that if we fix a subinterval $I\subseteq [0,\pi/2)$, then as $X\to \infty$,
\begin{equation}\label{asymp number of angles}
  \#\{ \mathfrak{p}\subset \Z[i]: X<\N(\mathfrak{p})\leq 2X,\ \theta_\mathfrak{p}\in I  \} \sim \frac{|I|}{\pi/2} \ \frac{X}{\log X}.
\end{equation}
Recall that by the Prime Ideal Theorem, $\#\{ \mathfrak{p}\subset \Z[i]: X<\N(\mathfrak{p})\leq 2X\}\sim X/\log X$.

Kubilius \cite{kubilius1955problem} and others studied the existence of prime angles in  narrow sectors.    Ricci  \cite{Ricci}  proved that \eqref{asymp number of angles}  remains valid
 for any interval $I\subset [0,\pi/2)$ of length $|I|>X^{-3/10+\varepsilon}$.
By a sieve method, Harman and Lewis \cite{harman2001gaussian} proved the existence of prime angles in somewhat narrower sectors without an asymptotic expression.
 Assuming the generalized Riemann Hypothesis (GRH), it is known that \eqref{asymp number of angles} is valid for any interval $I\subset [0,\pi/2)$ of length $|I|>X^{-1/2 +\varepsilon}$.
It is important to point out that one cannot do better because of the existence of ``forbidden regions'', for instance the interval $(0,1/\sqrt{X})$ does not contain any prime angle of norm less that $X$ (in fact angles for any such integer ideals).
This is a striking difference from the well studied problem of (rational) primes in short intervals, where one expects that any interval $[x,x+y]$ will contain $\sim y/\log x$ primes as soon as $y\gg x^\varepsilon$, for any $\varepsilon>0$, while similarly to  our problem, the Riemann Hypothesis only gives the existence in intervals of length $y\gg x^{1/2+o(1)}$.

Instead of asking about all short sectors, one can ask for the existence of prime angles in ``typical''  sectors, that is in almost all short sectors.
Assuming GRH,  Parzanchevski and Sarnak \cite{PS2017super}, and Rudnick and Waxman \cite{rudnick2017angles} showed that almost all sectors contain a prime angle, in fact that the asymptotic formula \eqref{asymp number of angles} (at least in a smooth form)  remains valid
for almost all $I$ if $|I|>X^{-1+\varepsilon}$, which is the most we can expect (up to log factors) since the number of prime ideals with norm about $X$ is roughly $  X/\log X$, hence  almost all sectors shorter than $1/X$ will contain no prime angles $\theta_{\mathfrak p}$ with $\N(\mathfrak p)\approx  X$.
Rudnick and Waxman \cite{rudnick2017angles}  gave a precise conjecture about the asymptotic behavior of the number variance, supported by a theorem for a function field analogue of the problem.

\medskip
In this paper, we prove an unconditional  result on Gaussian primes in almost all narrow sectors.
We will say that a property holds for \emph{almost all} narrow sectors $I=(\beta,\beta+\gamma]$ (where $\gamma\approx X^{-\rho}$, $0<\rho<1$, and $X<\N(\mathfrak{p})\leq 2X$ with $X\rightarrow\infty$) if it holds for all $\beta\in[0,\frac{\pi}{2}]$ except on an exceptional set of measure $o(1)$.

\begin{theorem}\label{thm:a.a.}
    Let $0 \leq \rho < 3/5$.         Then \eqref{asymp number of angles} holds for almost all short sectors $I\subset[0,\pi/2)$ of length $ |I|>X^{-\rho}$.
%
\end{theorem}
Our result is a consequence of an upper bound of the number variance estimated using zero-density theorems.

%

 We recall that under RH, Selberg \cite{selberg1943normal} showed (in 1943)
that almost all intervals $(x, x + \phi(x) \log^2 x]$ contain primes for any function $\phi(x)$ tending to infinity. See   \cite{SafVau},    \cite{jia1996almost},   \cite[Chap. 9]{harman2012prime}) for unconditional results.

The problem of Gaussian primes in small balls  is   similar in flavour to that of  primes in short intervals, and was studied by Coleman  \cite{coleman1990distribution} who established individual results, both on GRH and unconditionally (see also  \cite{harman2003distribution}), and a result about almost all balls.

We briefly discuss the analogous problems for real quadratic fields in Section~\ref{sec:real quadratic}.

\section{Preliminaries}

\subsection{Hecke characters and their L-functions}\label{subsec:L-functions}
For a non-zero ideal $\mathfrak a=(\alpha) \subseteq \Z[i]$, with generator $\alpha$,  Hecke defined characters $\Xi_k(\alpha) = (\alpha/\bar\alpha)^{2k}$, $k\in \mathbb Z$, which give well defined functions on the ideals of $\mathbb Z[i]$.
In terms of the angles associated to ideals, we have
$e^{i4k\theta_{\mathfrak a}} = \Xi_k(\mathfrak a)$.

To each such character Hecke \cite{Hecke} associated its L-function
$$
  L(s,\Xi_k) := \sum_{0\neq \mathfrak a\subseteq \mathbb Z[i]}
  \frac{\Xi_k(\mathfrak a)}{\N(\mathfrak a)^{s}}
  = \prod_{\substack{\mathfrak p\\{\rm prime}}}
  (1-\Xi_k(\mathfrak p)\N(\mathfrak p)^{-s})^{-1},
  \quad \Re(s)>1.
$$
Note that $L(s,\Xi_k) = L(s,\Xi_{-k})$. Hecke showed that if $k\neq 0$,
these functions have an analytic continuation to the entire complex plane,
and satisfy  a functional equation: 
\begin{equation}\label{functional equation}
  \xi(s,k):=\pi^{-(s+2|k|)}\Gamma(s+2|k|) L(s,\Xi_k) =\xi(1-s,k).
\end{equation}
We denote $L_\infty(s,k):=\pi^{-(s+2|k|)}\Gamma(s+2|k|).$

\subsection{The zero-free region}
We will need a ``non-standard'' zero free region for $L(s,\Xi_{k})$:
\begin{theorem}[Kubilius \cite{kubilius1955problem}]\label{eqn:zero-free}
  For $k>0$, and $V:=\sqrt{(T+2)^2+(2k)^2}$, if $\rho_{k,n}=\beta_{k,n}+i\gamma_{k,n}$ is a zero of $L(s,\Xi_{k})$, then we have
\begin{equation*}
1-\beta_{k,n}\gg \frac 1{(\log V \log\log V)^{3/4}}, \quad |\gamma_{k,n}|<T.
\end{equation*}
\end{theorem}

\subsection{The zero-density estimate}

We now introduce a zero-density theorem. We set
$$
N(\sigma;T,K):=\#\{\rho_{k,n}=\beta_{k,n}+i\gamma_{k,n}: 0<k\leq K,\quad |\gamma_{k,n}|<T, \quad \beta_{k,n}\geq \sigma\}.
$$
In his thesis (1976), Ricci \cite{Ricci} showed
\begin{theorem}[\cite{Ricci}]\label{zdt}
For $\sigma\geq 1/2$, $K,T\geq 2$, and $T=o(K) $, we have
\begin{equation*}
N(\sigma;T,K) \ll  T K^{\frac{10}{3}(1-\sigma)} (\log K)^B
\end{equation*}
for some $B>0$.
\end{theorem}

\section{The number variance}

We wish to get an unconditional result on angles of Gaussian primes in almost all narrow sectors. We recall some definitions:
Pick $f\in C_c^\infty(\mathbb{R})$, which is even and real valued, and for $K\gg1$ define
\begin{equation}\label{def of FK}
F_K(\theta) := \sum_{j\in\mathbb{Z}} f\Big(\frac{K}{\pi/2} \big(\theta-j\frac{\pi}{2}\big)\Big).
\end{equation}
Let $\Phi\in  C_c^\infty(0,\infty)$ and  set
$$
  \psi_{K,X}(\theta)
  :=\sum_{\mathfrak a  } \Phi\Big(\frac{\N(\mathfrak a)}{X}\Big) \Lambda(\mathfrak a) F_K(\theta_{\mathfrak a}  -\theta),
$$
the sum over all powers of prime ideals, with the von Mangoldt function
$\Lambda(\mathfrak a)=\log \N(\mathfrak p)$ if $\mathfrak a=\mathfrak p^r$ is a power of a prime ideal $\mathfrak p$, and equal to zero otherwise.

By \cite[Lemma 3.1]{rudnick2017angles}, the expected value of $\psi_{K,X}(\theta)$ is
$$
  \E(\psi_{K,X}) = \langle \psi_{K,X} \rangle \sim \frac{X}{K} \int_{-\infty}^{\infty} f(x) \dd x \int_{0}^{\infty} \Phi(u) \dd u.
$$
Here for any function $H$ of $\theta$, we define the mean value
\[ \langle H \rangle := \frac{1}{\pi/2} \int_{0}^{\pi/2} H(\theta) \dd \theta.  \]
We wish to study the number variance
\[
  \Var(\psi_{K,X}) = \frac{1}{\pi/2}\int_{0}^{\pi/2} \Big| \psi_{K,X}(\theta) - \E(\psi_{K,X}) \Big|^2 \dd \theta.
\]

\begin{theorem}\label{thm:variance}
  If $K=X^\tau$ with $\tau<3/5$, then
$$
\frac{\Var(\psi_{K,X})}{(\E(\psi_{K,X}))^2} \ll (\log X)^{-A},  \quad \textrm{as }  X\to \infty.
$$
\end{theorem}
Consequently we get, that in this range, for almost all $\theta$, there is a prime ideal $\mathfrak{p}$ with $\N(\mathfrak p)\asymp X$ so that $|\theta-\theta_{\mathfrak p}|<1/K$.

\subsection{The proof of Theorem \ref{thm:variance}}
We have  \cite[Corollary 4.4]{rudnick2017angles}\footnote{there it is formulated assuming GRH.}
  \begin{equation*}
 \psi_{K,X}(\theta)-\ave{  \psi_{K,X} } =
-     \sum_{k\neq 0 } e^{-i4k\theta} \frac 1{ K}\^f(\frac {k}{ K})
\left(  \sum_{n} \tilde \Phi(\rho_{k,n})X^{\rho_{k,n}}
  +O\Big(\frac{ X^{1/2}\log K}{(\log X)^{100}}\Big) \right),
\end{equation*}
the inner sum over all nontrivial zeros $\rho_{k,n} = \beta_{k,n}+i\gamma_{k,n}$ of $L(s,\Xi_k)$.

Computing the mean square over $\theta$ and dividing by the square of the expected value, namely by $(X/K)^2$, gives
$$
 \frac{\Var(\psi_{K,X})}{(\E(\psi_{K,X}))^2} \ll
\frac{1}{X^2} \sum_{k\neq 0} |\^f(\frac {k}{ K}) |^2 \left|  \sum_{n} \tilde \Phi(\rho_{k,n})X^{\rho_{k,n}}
  +O\Big(\frac{X^{1/2} \log K}{(\log X)^{100}}\Big) \right|^2 .
$$

Below we will see that the contribution of the first term is $ O((\log X)^{-A})$.
The contribution of the second term $O\Big(\frac{X^{1/2} \log K}{(\log X)^{100}}\Big)^2 $ is $O(X^{-\epsilon})$ if $K<X^{1-\epsilon}$, so that modulo this fact we have
$$
\frac{\Var(\psi_{K,X})}{(\E(\psi_{K,X}))^2} \ll \frac 1{X^2} \sum_{k\neq 0} |\^f(\frac {k}{ K}) |^2 \left|  \sum_{n} \tilde \Phi(\rho_{k,n})X^{\rho_{k,n}}   \right|^2 + X^{-\epsilon}.
$$

Next, we note that up to a negligible error, we can truncate the sum over $k$ to $0<k<K^{1+\delta}$, for any fixed $0<\delta<1$, where we take $\delta>0$ so that
$$K<X^{\frac 35(1-2\delta)}.
$$
Indeed, since $\^f$ is rapidly decaying, using the trivial bound $\Re \rho_{k,n}\leq 1$, we see that the tail of the sum is bounded by
$$
\frac 1{X^2} \sum_{ k>K^{1+\delta}} |\^f(\frac {k}{ K}) |^2 \left|  \sum_{n} \tilde \Phi(\rho_{k,n})X^{\rho_{k,n}}     \right|^2
\ll  \sum_{ k>K^{1+\delta}} |\^f(\frac {k}{ K}) |^2 \Big( \sum_{n} |\tilde \Phi(\rho_{k,n})| \Big)^2
$$
For each $k\neq 0$, we use the rapid decay of the Mellin transform $|\tilde \Phi(\rho_{k,n})| \ll (1+|\gamma_{k,n}|)^{-100 }$ and the bound
$$\#\{ n: T\leq |\gamma_{k,n}|<T+1\}\ll \log( 2|k|(T+2))$$
to bound the sum over the zeros by
$$
\sum_{n} |\tilde \Phi(\rho_{k,n})|\leq \sum_{T=0}^\infty \log(2|k|) \frac{\log(T+2)}{(1+T)^{100 }}\ll \log(2|k|)
$$
Hence the tail of the sum is bounded by
$$
\frac 1{X^2} \sum_{ k>K^{1+\delta}} |\^f(\frac {k}{ K}) |^2 \left|  \sum_{n} \tilde \Phi(\rho_{k,n})X^{\rho_{k,n}}     \right|^2
\ll  \sum_{ k>K^{1+\delta}} |\^f(\frac {k}{ K}) |^2 (\log k)^2 \ll K^{-100}
$$
on using $|\^f(y)|\ll |y|^{-100/\delta}$ for $|y|>1$. Therefore
$$
\frac{\Var(\psi_{K,X})}{(\E(\psi_{K,X}))^2} \ll \frac 1{X^2} \sum_{0<k<K^{1+\delta}} |\^f(\frac {k}{ K}) |^2 \left|  \sum_{n} \tilde \Phi(\rho_{k,n})X^{\rho_{k,n}}     \right|^2 +  X^{-\epsilon}.
$$

We may also restrict the sum to zeros with $|\gamma_{k,n}|< K^\varepsilon$, at the cost of another negligible error, again by using the rapid decay of $\tilde \Phi$. Thus we have
$$
\frac{\Var(\psi_{K,X})}{(\E(\psi_{K,X}))^2} \ll W+X^{-\epsilon}
$$
where
$$
W:= \frac 1{X^2} \sum_{0<k<K^{1+\delta}} |\^f(\frac {k}{ K}) |^2 \bigg|  \sum_{n:|\gamma_{k,n}|<K^\varepsilon} \tilde \Phi(\rho_{k,n})X^{\rho_{k,n}} \bigg|^2.
$$
By the Cauchy--Schwarz inequality,
$$
W\leq  \frac 1{X^2} \sum_{0<k<K^{1+\delta}} |\^f(\frac {k}{ K}) |^2   \sum_{n:|\gamma_{k,n}|<K^\varepsilon}
| \tilde \Phi(\rho_{k,n})|  \sum_{n:|\gamma_{k,n}|<K^\varepsilon}| \tilde \Phi(\rho_{k,n})| X^{2\beta_{k,n}}.
$$
Using  the rapid decay of $\tilde \Phi$ and the logarithmic density of zeros gives
$$
\sum_{n:|\gamma_{k,n}|<K^\varepsilon}
| \tilde \Phi(\rho_{k,n})|   \ll \log K
$$
so that
$$
W\ll   \frac{\log K}{X^2} \sum_{0<k<K^{1+\delta}} |\^f(\frac {k}{ K}) |^2    \sum_{n:|\gamma_{k,n}|<K^\varepsilon}| \tilde \Phi(\rho_{k,n})| X^{2\beta_{k,n}}.
$$

Let
$$M= (\log K)^{9/10} ,\qquad  \Delta = \lceil \frac {2}{\delta} \rceil
$$
so that for $K\gg 1$, there are no zeros of $L(s,\Xi_k)$ with real part $\beta_{k,n}>1-\frac{\Delta}{M}$,
by Theorem~\ref{eqn:zero-free}.
We use a dyadic decomposition to bound $W$:
\begin{equation*}
\begin{split}
W&\ll \frac {\log K} {X^2} \sum_{0<k<K^{1+\delta}}  \sum_{m=0}^{M-\Delta  }
\sum_{0\leq T<K^\varepsilon}
\sum_{\substack{ \frac 12 +\frac{m }{2M}\leq \beta_{k,n}<\frac 12 +\frac{m+1}{2M} \\ T\leq |\gamma_{k,n}|<T+1}}   | \tilde \Phi(\rho_{k,n}) |  X^{2\beta_{k,n}}
\\
&\ll  \frac  {\log K}{X^2} \sum_{0<k<K^{1+\delta}}     \sum_{m=0}^{M-\Delta   }
\sum_{0\leq T<K^\varepsilon}
\sum_{\substack{ \frac 12 +\frac{m }{2M}\leq \beta_{k,n}<\frac 12 +\frac{m+1}{2M} \\ T\leq |\gamma_{k,n}|<T+1}}
X^{1 + \frac{m+1}{ M}} \frac{1}{(T+1)^{100}}
\\
&\ll
\frac  {\log K}{X^2} \sum_{m=0}^{M-\Delta   }X^{1 + \frac{m+1}{ M}}  \sum_{0\leq T<K^\varepsilon} \frac{1}{(T+1)^{100}}   \Big(
 \sum_{0<k<K^{1+\delta}}\sum_{\substack{ \frac 12 +\frac{m }{2M}\leq \beta_{k,n}<\frac 12 +\frac{m+1}{2M} \\ T\leq \gamma_{k,n}<T+1}} 1
\Big)
\\
&\leq
\frac  {X^{1/M} \log K}{X } \sum_{m=0}^{M-\Delta   }X^{  \frac{m }{ M}}  \sum_{0\leq T<K^\varepsilon} \frac{1}{(T+1)^{100}}
N(\frac 12 +\frac m{2M};T,K^{1+\delta}).
\end{split}
\end{equation*}

Now use the zero-density theorem (Theorem~\ref{zdt}) to bound, for $T<K^\varepsilon$,
$$
N(\frac 12 +\frac m{2M};T,K^{1+\delta}) \ll T K^{(1+\delta)\frac{10}{3}(\frac 12-\frac{m}{2M}) } (\log K)^B.
$$
Hence
\begin{equation*}
\begin{split}
W&\ll
\frac  {X^{1/M} (\log K)^{B+1}}{X } \sum_{m=0}^{M-\Delta   }X^{  \frac{m }{ M}}  \sum_{1\leq T<K^\varepsilon} \frac{T}{T^{100}}
K^{(1+\delta)\frac{10}{3}(\frac 12-\frac{m}{2M}) }
\\
&\ll
\frac  {X^{1/M} (\log K)^{B+1}}{X } \sum_{m=0}^{M-\Delta   }X^{  \frac{m }{ M}}K^{(1+\delta)\frac{10}{3}(\frac 12-\frac{m}{2M})}
\\
&\ll
\frac  { K^{(1+\delta)\frac{5}{3}} }{X }X^{1/M} (\log K)^{B+1}
 \sum_{m=0}^{M-\Delta   } \Big(  \frac{X} {K^{(1+\delta)\frac 53}} \Big)^{\frac{m}{M}}.
\end{split}
\end{equation*}
Since we assume that $K<X^{\frac 35(1-2\delta)}$ (here $\delta>0$ is arbitrarily small),   so that
$$
\frac{X} {K^{(1+\delta)\frac 53}} >X^{1-(1+\delta)\frac 53 (1-2\delta )\frac 35} = X^{\delta+2\delta^2}>X^\delta
$$
we have
$$
 \sum_{m=0}^{M-\Delta   } \Big(  \frac{X} {K^{(1+\delta)\frac 53}} \Big)^{\frac{m}{M}} \ll
 \Big(  \frac{X} {K^{(1+\delta)\frac 53}} \Big)^{1-\frac {\Delta}M}
$$
so that we obtain
$$
W\ll X^{1/M}(\log K)^{B+1} \Big(  \frac{X} {K^{(1+\delta)\frac 53}} \Big)^{ -\frac {\Delta}M}
\ll
\frac{ (\log K)^{B+1}}{X^{\frac{\delta \Delta-1}{M}}}.
$$

Since we took $\Delta = \lceil 2/\delta \rceil$, we have $\frac{\delta \Delta-1}{M}\geq \frac 1M$, so that
\begin{multline*}
W\ll \frac{ (\log K)^{B+1}}{X^{\frac{ 1}{M}}} \ll   (\log K)^{B+1} \exp(-  \log X/(\log K)^{9/10})
\\
\ll \exp(- \frac 12 (\log X)^{1/10} )
= O((\log X)^{-A}).
\end{multline*}
 Thus we see that if $K<X^{\frac 35(1-2\delta)}$,
$$
 \frac{\Var(\psi_{K,X})}{(\E(\psi_{K,X}))^2} \ll (\log X)^{-A}.
$$
This completes the proof of Theorem \ref{thm:variance}.

\subsection{Primes vs prime powers}

Now set
\[
  \psi_{K,X}^{\mathrm{prime}}(\theta)
  :=\sum_{\mathfrak p} \Phi\Big(\frac{\N(\mathfrak p)}{X}\Big) F_K(\theta_{\mathfrak p}  -\theta) \log \N(\mathfrak p),
\]
the sum over all prime ideals. Let
\[
  \Var(\psi_{K,X}^{\mathrm{prime}}) = \frac{1}{\pi/2}\int_{0}^{\pi/2} \Big| \psi_{K,X}^{\mathrm{prime}}(\theta) - \E(\psi_{K,X}^{\mathrm{prime}}) \Big|^2 \dd \theta.
\]

\begin{corollary}\label{cor:var-prime}
  If $K=X^\tau$ with $\tau<3/5$, then
  \[
    \frac{\Var(\psi_{K,X}^{\mathrm{prime}})}{(\E(\psi_{K,X}^{\mathrm{prime}}))^2}
    \ll (\log X)^{-A}, \quad \textrm{as } X\to \infty.
  \]
\end{corollary}

\begin{proof}
By \cite[Lemma 3.1]{rudnick2017angles} again, the mean value of $\psi_{K,X}^{\mathrm{prime}}(\theta)$ is
\begin{equation}\label{eqn:<phi>}
   \E(\psi_{K,X}^{\mathrm{prime}})
   = \langle \psi_{K,X}^{\mathrm{prime}} \rangle \sim
   \langle \psi_{K,X} \rangle
   \sim \frac{X}{K} \int_{-\infty}^{\infty} f(x) \dd x \int_{0}^{\infty} \Phi(u) \dd u.
\end{equation}
Moreover,
\[
  |\langle \psi_{K,X}^{\mathrm{prime}} \rangle -
   \langle \psi_{K,X} \rangle | \ll \frac{X^{1/2}}{K}.
\]
And by the prime ideal theorem, we have
\[
  \sum_{\mathfrak a} \Phi\Big(\frac{\N(\mathfrak a)}{X}\Big) \Xi_k(\mathfrak a) \Lambda(\mathfrak a)
  = \sum_{\mathfrak p} \Phi\Big(\frac{\N(\mathfrak p)}{X}\Big) \Xi_k(\mathfrak p) \log \N(\mathfrak p) + O(X^{1/2}).
\]
Then by the same proof as in \cite[Lemma 4.8]{rudnick2017angles}, we have
\[
  \langle |\psi_{K,X} - \psi_{K,X}^{\mathrm{prime}}|^2 \rangle
  \ll \frac{X}{K}.
\]
Using the triangle inequality as in \cite[\S4.4]{rudnick2017angles}, together with the above estimates and Theorem \ref{thm:variance}, we complete the proof of our corollary.
\end{proof}

\subsection{Proof of Theorem~\ref{thm:a.a.}}
By Chebyshev's inequality and Corollary \ref{cor:var-prime}, we find that
\begin{equation}\label{eqn:psi_beta}
 \psi_{K,X}^{\mathrm{prime}}(\beta)\sim \E(\psi_{K,X}^{\mathrm{prime}}) \sim
     \frac{X}{K} \int_{-\infty}^{\infty} f(x) \dd x \int_{0}^{\infty} \Phi(u) \dd u.
\end{equation}
for almost all $\beta\in [0,\pi/2)$.


Theorem~\ref{thm:a.a.} follows by removing the smoothing.
Let $f_\varepsilon^+(y)$ be a smooth function depending on $\varepsilon$ satisfying that $f_\varepsilon^+(x)=1$ if $x\in[0,1]$, $f_\varepsilon^+(x)\in[0,1]$ if $x\in[-\varepsilon,0]\cup[1,1+\varepsilon]$, and
$f_\varepsilon^+(x)=0$ otherwise.
Here $\varepsilon$ is a small positive constant.
Define  $F_{K,\varepsilon}^+(\theta)$ as in \eqref{def of FK},  that is
\[
F_{K,\varepsilon}^+(\theta) = \sum_{j\in\mathbb{Z}} f_\varepsilon^+\Big(\frac{K}{\pi/2} \big(\theta-j\frac{\pi}{2}\big)\Big).
 \]
Similarly, we can choose $\Phi_\varepsilon^+$ to be a smooth function depending on $\varepsilon$ satisfying that $\Phi_\varepsilon^+(u)=1$ if $u\in[1,2]$, $\Phi_\varepsilon^+(u)\in[0,1]$ if $u\in[1-\varepsilon,1]\cup[2,2+\varepsilon]$, and
$\Phi_\varepsilon^+(u)=0$ otherwise.

Let
\[
  \begin{split}
     \mathcal A & = \{ \mathfrak{p}\subseteq \mathbb{Z}[i]: \theta_\mathfrak{p}\in I:= [\beta, \beta+\frac{\pi/2}{K}],\ X<\N(\mathfrak{p})\leq 2X \}.
  \end{split}
\]
Note that for $X>2^{1/\varepsilon}$,
\begin{equation*}
  \sum_{\mathfrak{p}\in\mathcal{A}} 1 \leq
  \sum_{\mathfrak{p}}  \Phi_\varepsilon^+\Big(\frac{\N(\mathfrak{p})}{X}\Big) F_{K,\varepsilon}^+(\theta_\mathfrak{p}-\beta)
  \leq \frac{1+\varepsilon}{\log X}\sum_{\mathfrak{p}}  \Phi_\varepsilon^+\Big(\frac{\N(\mathfrak{p})}{X}\Big) F_{K,\varepsilon}^+(\theta_\mathfrak{p}-\beta) \log \N(\mathfrak{p}).
\end{equation*}
For $\beta\in[0,\pi/2)$ satisfying \eqref{eqn:psi_beta} with $f=f_\varepsilon^+$ and $\Phi=\Phi_\varepsilon^+$, we have
\begin{equation*}
  \sum_{\mathfrak{p}\in\mathcal{A}} 1
  \leq (1+2\varepsilon)^4 \frac{X}{K\log X}
  = (1+2\varepsilon)^4\frac {|I|}{\pi/2} \frac{X}{\log X},
\end{equation*}
for $X>X(\varepsilon)$ sufficiently large.

Similarly, we can obtain a lower bound of $\sum_{\mathfrak{p}\in\mathcal{A}} 1$ by another smooth counting, getting that for $\beta\in[0,\pi/2)$ satisfying \eqref{eqn:psi_beta} with $f=f_\varepsilon^-$ and $\Phi=\Phi_\varepsilon^-$, we have
\begin{equation*}
  \sum_{\mathfrak{p}\in\mathcal{A}} 1  \geq
  (1-2\varepsilon)^4\frac {|I|}{\pi/2} \frac{X}{\log X},
\end{equation*}
for $X>X(\varepsilon)$.
Taking the limit $\varepsilon\rightarrow0$, we obtain an asymptotic formula for $\sum_{\mathfrak{p}\in\mathcal{A}} 1$ for almost all $\beta$. This completes the proof of Theorem \ref{thm:a.a.}.


\section{Prime angles for real quadratic fields} \label{sec:real quadratic}
One can also study questions about angles related to the representation of primes as norms in a real quadratic field, similar to those for Gaussian primes.

 Let $E$  be a real quadratic field, and $\epsilon=\epsilon_E>1$ the fundamental unit, so that the group of units is
$\  \{\pm \epsilon^n:n\in \Z\}$.  
For simplicity, we will assume that the fundamental unit has negative norm, and that the (narrow) class number of $E$   is one, so that all ideals are principal (Gauss conjectured that this occurs infinitely often).
Given a prime $p$ which splits (completely) in $E$, to any solution of $\Norm_{E/\Q}(\alpha) =   p$,
that is to any generator of an ideal $\mathfrak p\mid (p)$, we associate
an angle variable
$$t(\mathfrak p) \in \R/(2\log\epsilon)\Z\simeq S^1  $$
by
  $$
 \exp( \frac{i\pi t(\mathfrak p)}{\log \epsilon})= |\frac \alpha{\tilde \alpha}|^{\frac{i\pi}{\log \epsilon}}
$$
(where $\alpha\mapsto \tilde \alpha$ is the Galois involution of $E$),
which is independent on the choice of generator of the ideal $\mathfrak p$. Note that for the Galois conjugate ideal we have $t(\tilde{\mathfrak p}) = -t(\mathfrak p) \bmod 2\log \epsilon$.

 For example, take   $E=\Q(\sqrt{2})$. The ring of integers is $O_E=\Z[\sqrt{2}]$,  the fundamental unit $\epsilon=1+\sqrt{2}$ has negative norm, and the class number (both wide and narrow) is one.
 The split primes are those satisfying $p=\pm 1\bmod 8$. For every such prime,   we can represent both $\pm p$ as a norm, that is solve $a^2-2b^2=\pm p$. The corresponding angle parameters describe the relative size of the solution coordinates $a$ and $b$.

 Hecke \cite{Hecke} showed that as $p$ varies over split primes, the corresponding angle parameters become uniformly distributed in $\R/(2\log\epsilon)\Z\simeq S^1$.

Using ideas similar to those for the case of Gaussian primes, one can show results analogous to Theorem~\ref{thm:a.a.}.
The non-standard zero-free region for the corresponding L-functions are due to Coleman \cite{coleman1990zero}.
The zero-density theorem needed can be proved along the lines of Theorem~\ref{zdt}.


\end{document}